\documentclass[11pt]{article}\UseRawInputEncoding
\usepackage{amssymb,amsfonts,amsmath,latexsym,epsf,tikz,url}
\usepackage{graphics}
\usepackage[usenames,dvipsnames]{pstricks}
\usepackage{pstricks-add}
\usepackage{epsfig}
\usepackage{pst-grad} 
\usepackage{pst-plot} 
\usepackage[space]{grffile} 
\usepackage{etoolbox} 
\makeatletter 
\patchcmd\Gread@eps{\@inputcheck#1 }{\@inputcheck"#1"\relax}{}{}
\makeatother

\newtheorem{theorem}{Theorem}[section]
\newtheorem{proposition}[theorem]{Proposition}
\newtheorem{conjecture}[theorem]{Conjecture}
\newtheorem{corollary}[theorem]{Corollary}

\newtheorem{remark}[theorem]{Remark}

\newcommand{\qed}{\hfill $\square$\medskip}

\begin{document}

\title{Co-even domination number of some binary operations on graphs}

\author{
Nima Ghanbari
}

\date{\today}

\maketitle

\begin{center}
Department of Informatics, University of Bergen, P.O. Box 7803, 5020 Bergen, Norway\\
\bigskip

{\tt Nima.Ghanbari@uib.no }
\end{center}


\begin{abstract}
Let $G=(V,E)$ be a simple graph. A dominating set of $G$ is a subset $D\subseteq V$ such that every vertex not in $D$ is adjacent to at least one vertex in $D$.
The cardinality of a smallest dominating set of $G$, denoted by $\gamma(G)$, is the domination number of $G$. A dominating set $D$ is called co-even dominating set if the degree of vertex $v$ is even number for all $v\in V-D$. 
 The cardinality of a smallest co-even dominating set of $G$, denoted by $\gamma _{coe}(G)$, is the co-even domination number of $G$.
In this paper, we study the co-even domination number of some binary operations on graphs.
\end{abstract}


\medskip
\noindent{\bf AMS Subj.\ Class.:}  05C69, 05C76

\section{Introduction}

Let $G = (V,E)$ be a simple graph with $n$ vertices. Throughout this paper we consider only simple graphs.  A set $D\subseteq V(G)$ is a  dominating set if every vertex in $V(G)- D$ is adjacent to at least one vertex in $D$.
The  domination number $\gamma(G)$ is the minimum cardinality of a dominating set in $G$. There are various domination numbers in the literature.
For a detailed treatment of domination theory, the reader is referred to \cite{domination}.

\medskip

Recently, Shalaan et. all introduced the concept of co-even domination number \cite{Sha}. By their definition,
a dominating set $D$ is called a co-even dominating set if the degree of vertex $v$ is even number for all $v\in V-D$. The cardinality of a smallest co-even dominating set of $G$, denoted by $\gamma _{coe}(G)$, is the co-even domination number of $G$.  They studied the co-even domination number of specific graphs such as path, cycle, complete, complete bipartite, star, regular, and wheel graphs. Later, they studied the co-even domination number of ladder, lollipop, butterfly, jellyfish, helm, corona, fan, and double fan graphs in \cite{Sha1}. Demirpolat et. all in \cite{Demir}, presented  co-even domination number of path related graphs as thorn graphs, thorn path, thorn rod, thorn
ring, thorn star, banana tree, coconut tree and binomial trees.

\medskip 

The join $G = G_1 + G_2$ of two graph $G_1$ and $G_2$ with disjoint vertex sets $V_1$ and $V_2$ and
edge sets $E_1$ and $E_2$ is the graph union $G_1\cup G_2$ together with all the edges joining $V_1$ and
$V_2$ \cite{Har}.The corona $G_1\circ G_2$ is the graph arising from the
disjoint union of $G_1$ with $| V_1 |$ copies of $G_2$, by adding edges between
the $i$th vertex of $G_1$ and all vertices of $i$th copy of $G_2$ \cite{Har}. The corona $G\circ K_1$, in particular, is the graph constructed from a copy of $G$,
where for each vertex $v\in V(G)$, a new vertex $v'$ and a pendant edge $vv'$ are added. The neighbourhood corona  $G_1 \star G_2$ is the graph obtained by
taking one copy of $G_1$ and $|V_1|$ copies of $G_2$ and  joining the neighbours of the $i$th vertex of $G_1$ to every vertex in the $i$th copy of $G_2$ \cite{Gop}. Let $G$ and $G'$ be two undirected graphs, $vw$ be an edge of $G$, and $xy$ be an edge of $G'$. Then the Haj\'{o}s sum $G''=G(vw)+_H G'(xy),$ forms a new graph that combines the two graphs by identifying vertices $v$ and $x$ into a single vertex, removing the two edges $vw$ and $xy$, and adding a new edge $wy$ \cite{HAJOSSUM}.

\medskip
In the next Section, we study the co-even domination number of join and corona of two graphs. In Section 3, we find some bounds for the neighbourhood corona of two graphs. Finally, in Section 4, we find a sharp upper bound for the Haj\'{o}s sum of two graphs and propose a conjecture for the lower bound.

\section{Co-even domination number of join and corona of two graphs}

In this section, we study the co-even domination number of join and corona two graphs. First we state some known results.

	\begin{proposition}\cite{Sha}\label{pro-sha}
Let $G=(V,E)$ be a graph and $D$ is a co-even dominating set. Then,
\begin{itemize}
\item[(i)]
All vertices of odd or zero degrees belong to every co-even dominating set.
\item[(ii)]
$deg(v)\geq 2$, for all $v\in V-D$.
\item[(iii)]
If $G$ is $r$-regular graph, then
	 \[
 	\gamma _{coe}(G)=\left\{
  	\begin{array}{ll}
  	{\displaystyle
  		n}&
  		\quad\mbox{if $r$ is odd, }\\[15pt]
  		{\displaystyle
  			\gamma (G)}&
  			\quad\mbox{if $r$ is even.}
  				  \end{array}
  					\right.	
  					\]
\item[(iv)]
$\gamma (G) \leq \gamma_{coe} (G).$
\end{itemize}
	\end{proposition}

Now we consider to join of two graphs:

	\begin{theorem}
Let $G=(V(G),E(G))$ and $H=(V(H),E(H))$ be two connected graphs, $E_{V(G)}$ and $O_{V(G)}$ be the set of vertices with even degree and odd degree of $G$, respectively. Then for $G+H$, 
\begin{itemize}
\item[(i)]
If $|V(G)|=2n$ and $|V(H)|=2n'$, for some $n,n'\in \mathbb{N} $, then if we have at least one vertex with odd degree in both of graphs, 
$$\gamma _{coe}(G+H)=\Big|O_{V(G)}\cup O_{V(H)}\Big|,$$
otherwise, $\gamma _{coe}(G+H) \leq 2$.
\item[(ii)]
If $|V(G)|=2n-1$ and $|V(H)|=2n'-1$, for some $n,n'\in \mathbb{N} $, then if we have at least one vertex with even degree in both of graphs, 
$$\gamma _{coe}(G+H)=\Big|E_{V(G)}\cup E_{V(H)}\Big|,$$
otherwise, $\gamma _{coe}(G+H) \leq 2$.
\item[(iii)]
If $|V(G)|=2n$ and $|V(H)|=2n'-1$, for some $n,n'\in \mathbb{N} $, then if we have at least one vertex with even degree in $G$ and at least one vertex with odd degree in $H$, 
$$\gamma _{coe}(G+H)=\Big|E_{V(G)}\cup O_{V(H)}\Big|,$$
otherwise, $\gamma _{coe}(G+H) \leq 2$.
\end{itemize}
	\end{theorem}

	\begin{proof}
\begin{itemize}
\item[(i)]
Suppose that $|V(G)|=2n$ and $|V(H)|=2n'$, for some $n,n'\in \mathbb{N} $ and we have at least one vertex with odd degree in both of graphs such as $g\in V(G)$ and $h\in V(H)$. Then by definition of join of two graphs, the degree of $g$ and $h$ in $G+H$ is odd too. By Proposition \ref{pro-sha}, all vertices of odd degrees belong to every co-even dominating set. It is easy to see that $O_{V(G)}\cup O_{V(H)}$ is a co-even dominating set for $G+H$. Now suppose that $G$ has only one vertex with odd degree, $g'\in V(G)$, and all vertices in $H$ have even degree. Then by letting our dominating set as union of $\{g'\}$ and one arbitrary vertex of $H$, we have a co-even dominating set for $G+H$. In case that $\gamma _{coe}(G)=1$, then we can choose that vertex for $G+H$ too. So $\gamma _{coe}(G+H) \leq 2$, and therefore we have the result.
\item[(ii)]
The proof is similar to part (i).
\item[(iii)]
The proof is similar to part (i).\qed
\end{itemize}
	\end{proof}

Now we find the co-even domination number of  the corona of two graphs.

	\begin{theorem}\label{GoH}
Let $G=(V(G),E(G))$ and $H=(V(H),E(H))$ be two connected graphs, $E_{V(G)}$ and $O_{V(G)}$ be the set of vertices with even degree and odd degree of $G$, respectively. Also $H$ is not graph $K_1$. Then for $G\circ H$, 
\begin{itemize}
\item[(i)]
If for every $u\in V(H)$, $d(u)=2n-1$, for some $n\in \mathbb{N} $, then
$$\gamma _{coe}(G\circ H)=\Big|V(G)\Big|.$$
\item[(ii)]
If there exists at least one vertex with even degree in $H$, then If $|V(H)|=2t$, for some $t\in \mathbb{N} $, we have 
$$ \gamma _{coe}(G\circ H)= \Big|V(G)\Big|\Big|E_{V(H)}\Big|+\Big|O_{V(G)}\Big|,$$
otherwise, 
$$ \gamma _{coe}(G\circ H)= \Big|V(G)\Big|\Big|E_{V(H)}\Big|+\Big|E_{V(G)}\Big|.$$
\end{itemize}
	\end{theorem}

	\begin{proof}
\begin{itemize}
\item[(i)]
Suppose that the degree of all vertices in $H$ are odd. Then by the definition of corona of two graphs, the degree of these vertices in $G \circ H$ are even. Therefore by considering our dominating set as $V(G)$, we have $\gamma _{coe}(G\circ H)\leq \Big|V(G)\Big|$. One can easily check that we do not have a smaller size choice for our dominating set. Therefore we have the result.
\item[(ii)]
Suppose that the degree of some vertices in $H$ are even. Then  the degree of these vertices in $G \circ H$ are odd. By Proposition \ref{pro-sha}, we need all these vertices in our dominating set. Every vertex in $G$ is now dominated by at least one vertex in a copy of $H$. Therefore the union of vertices with even degree of copies of $H$ is a subset of our co-even secure dominating set. Now, if $|V(H)|=2t$, for some $t\in \mathbb{N} $, then the degree of an odd vertex in $G$ will remain odd in $G \circ H$ and we need these vertices in our dominating set too. Since we should put all vertices with odd degree in our dominating set, then the mentioned set is a secure dominating set with smallest size and
 $$ \gamma _{coe}(G\circ H)= \Big|V(G)\Big|\Big|E_{V(H)}\Big|+\Big|O_{V(G)}\Big|.$$ 
 By the same argument, if $|V(H)|=2t'-1$, for some $t'\in \mathbb{N} $, then 
$$ \gamma _{coe}(G\circ H)= \Big|V(G)\Big|\Big|E_{V(H)}\Big|+\Big|E_{V(G)}\Big|,$$ 
and therefore we have the result.
\qed
\end{itemize}
	\end{proof}

In Theorem \ref{GoH}, we considered all cases except $H=K_1$.  Now we consider to this special case.

	\begin{theorem}\label{GoK1}
Let $G=(V(G),E(G))$ be a connected graph and $E_{V(G)}$  be the set of vertices with even degree  of $G$. Then for $G\circ K_1$, 
$$ \gamma _{coe}(G\circ K_1)= \Big|V(G)\Big|+\Big|E_{V(G)}\Big|.$$ 
	\end{theorem}

	\begin{proof}
Since every vertex in copies of $K_1$ has odd degree in $G\circ K_1$, then all these vertices should be in our dominating set. Now it is clear that this set of vertices is a dominating set for $G\circ K_1$. Every vertex with even degree in $G$ is a vertex with odd degree now, and we need these vertices in our dominating set too. So 
$$ \gamma _{coe}(G\circ K_1)= \Big|V(G)\Big|+\Big|E_{V(G)}\Big|.$$ 
\qed
	\end{proof}	

We end this section by an immediate  result of Theorem \ref{GoK1}.

\begin{corollary}
$ \gamma _{coe}(P_n\circ K_1)=2n-2$
and 
$ \gamma _{coe}(C_n\circ K_1)=2n$.
\end{corollary}

\section{ Co-even domination number of the neighbourhood corona of two graphs }

In this section, we consider to the neighbourhood corona of two graphs. 
The neighbourhood corona  $G_1 \star G_2$ is the graph obtained by
taking one copy of $G_1$ and $|V_1|$ copies of $G_2$ and  joining the neighbours of the $i$th vertex of $G_1$ to every vertex in the $i$th copy of $G_2$. Figure \ref{fig1} shows $P_4\star P_3$, where $P_n$ is  the path of order $n$. Now we propose an upper and lower bound for co-even domination number of the neighbourhood corona of two graphs. Also show that these bounds are sharp.

	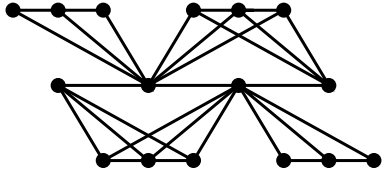
\begin{figure}
		\begin{center}
			\psscalebox{0.5 0.5}
{
\begin{pspicture}(0,-5.6)(9.994231,-1.2057691)
\psdots[linecolor=black, dotsize=0.4](0.19711533,-1.4028845)
\psdots[linecolor=black, dotsize=0.4](1.3971153,-1.4028845)
\psdots[linecolor=black, dotsize=0.4](2.5971153,-1.4028845)
\psdots[linecolor=black, dotsize=0.4](1.3971153,-3.4028845)
\psdots[linecolor=black, dotsize=0.4](3.7971153,-3.4028845)
\psdots[linecolor=black, dotsize=0.4](3.7971153,-5.4028845)
\psdots[linecolor=black, dotsize=0.4](2.5971153,-5.4028845)
\psdots[linecolor=black, dotsize=0.4](4.997115,-5.4028845)
\psdots[linecolor=black, dotsize=0.4](4.997115,-1.4028845)
\psdots[linecolor=black, dotsize=0.4](6.1971154,-1.4028845)
\psdots[linecolor=black, dotsize=0.4](7.397115,-1.4028845)
\psdots[linecolor=black, dotsize=0.4](6.1971154,-3.4028845)
\psdots[linecolor=black, dotsize=0.4](7.397115,-5.4028845)
\psdots[linecolor=black, dotsize=0.4](8.5971155,-5.4028845)
\psdots[linecolor=black, dotsize=0.4](9.797115,-5.4028845)
\psdots[linecolor=black, dotsize=0.4](8.5971155,-3.4028845)
\psline[linecolor=black, linewidth=0.04](1.3971153,-3.4028845)(1.3971153,-3.4028845)
\psline[linecolor=black, linewidth=0.08](1.3971153,-3.4028845)(8.5971155,-3.4028845)(8.5971155,-3.4028845)
\psline[linecolor=black, linewidth=0.08](1.3971153,-3.4028845)(4.997115,-5.4028845)(4.997115,-5.4028845)
\psline[linecolor=black, linewidth=0.08](1.3971153,-3.4028845)(3.7971153,-5.4028845)(3.7971153,-5.4028845)
\psline[linecolor=black, linewidth=0.08](1.3971153,-3.4028845)(2.5971153,-5.4028845)(2.5971153,-5.4028845)
\psline[linecolor=black, linewidth=0.08](2.5971153,-5.4028845)(4.997115,-5.4028845)(4.997115,-5.4028845)
\psline[linecolor=black, linewidth=0.08](8.5971155,-3.4028845)(4.997115,-1.4028845)(7.397115,-1.4028845)(8.5971155,-3.4028845)(6.1971154,-1.4028845)(6.5971155,-1.4028845)
\psline[linecolor=black, linewidth=0.08](0.19711533,-1.4028845)(2.5971153,-1.4028845)(3.7971153,-3.4028845)(1.3971153,-1.4028845)(1.3971153,-1.4028845)
\psline[linecolor=black, linewidth=0.08](0.19711533,-1.4028845)(3.7971153,-3.4028845)(4.997115,-1.4028845)(4.997115,-1.4028845)
\psline[linecolor=black, linewidth=0.08](3.7971153,-3.4028845)(6.1971154,-1.4028845)(6.1971154,-1.4028845)
\psline[linecolor=black, linewidth=0.08](3.7971153,-3.4028845)(7.397115,-1.4028845)(7.397115,-1.4028845)
\psline[linecolor=black, linewidth=0.08](6.1971154,-3.4028845)(2.5971153,-5.4028845)(2.5971153,-5.4028845)
\psline[linecolor=black, linewidth=0.08](6.1971154,-3.4028845)(3.7971153,-5.4028845)(3.7971153,-5.4028845)
\psline[linecolor=black, linewidth=0.08](6.1971154,-3.4028845)(4.997115,-5.4028845)(4.997115,-5.4028845)
\psline[linecolor=black, linewidth=0.08](6.1971154,-3.4028845)(7.397115,-5.4028845)(7.397115,-5.4028845)
\psline[linecolor=black, linewidth=0.08](7.397115,-5.4028845)(9.797115,-5.4028845)(6.1971154,-3.4028845)(8.5971155,-5.4028845)(8.5971155,-5.4028845)
\end{pspicture}
}
		\end{center}
		\caption{Graph $P_4 \star P_3$} \label{fig1}
	\end{figure}

	\begin{theorem}\label{GstarH}
Let $G=(V(G),E(G))$ and $H=(V(H),E(H))$ be two connected graphs, $E_{V(G)}$ and $O_{V(G)}$ be the set of vertices with even degree and odd degree of $G$, respectively. Then for $G\star H$, 
\begin{itemize}
\item[(i)]
$$ \gamma _{coe}(G\star H)\geq \Big|O_{V(G)}\Big|\Big|E_{V(H)}\Big|+\Big|E_{V(G)}\Big|\Big|O_{V(H)}\Big|.$$
\item[(ii)]
$$ \gamma _{coe}(G\star H)\leq \Big|V(G)\Big|+\Big|O_{V(G)}\Big|\Big|E_{V(H)}\Big|+\Big|E_{V(G)}\Big|\Big|O_{V(H)}\Big|.$$
\end{itemize}
	\end{theorem}

	\begin{proof}
 Suppose that $u$ is a vertex with odd degree and $v$ is a vertex with even degree in $G$.  By the definition of $G\star H$, the neighbours of each vertex of $G$ are connected to every vertex in the corresponding copy of $H$. Therefore the vertices in copy of $H$ corresponding to $u$ with odd degree change to a vertex with even one and the vertices  with even degree change to a vertex with odd one. Therefore, by Proposition \ref{pro-sha}, we need all vertices in $E_{V(H)}$ corresponding to $u$ in our co-even domination set. So for all vertices with odd degree in $G$ we have the same and the number of these vertices is $\big|O_{V(G)}\big|\big|E_{V(H)}\big|$. Also, the vertices in copy of $H$ corresponding to $v$ with odd degree remain vertices with odd degree and the vertices  with even degree  remain vertices with even degree. Therefore, we need all vertices in $O_{V(H)}$ corresponding to $v$ in our co-even domination set. Hence for all vertices with even degree in $G$ we have the same and the number of these vertices is $\big|E_{V(G)}\big|\big|O_{V(H)}\big|$. So,
 $$ \gamma _{coe}(G\star H)\geq \Big|O_{V(G)}\Big|\Big|E_{V(H)}\Big|+\Big|E_{V(G)}\Big|\Big|O_{V(H)}\Big|.$$
 On the other hand, if we put all the vertices of $G$ in the mentioned set, then one can easily check that we have a co-even dominating set for $G\star H$. Therefore,
 $$ \gamma _{coe}(G\star H)\leq \Big|V(G)\Big|+\Big|O_{V(G)}\Big|\Big|E_{V(H)}\Big|+\Big|E_{V(G)}\Big|\Big|O_{V(H)}\Big|,$$
and we have the result.    \qed
	\end{proof}

	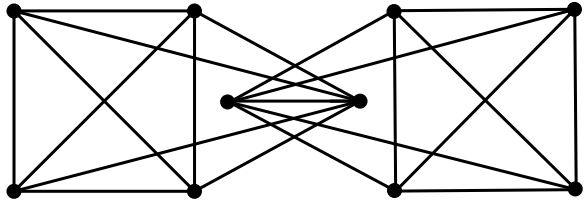
\begin{figure}
		\begin{center}
			\psscalebox{0.5 0.5}
{
\begin{pspicture}(0,-14.78)(15.314231,-9.545769)
\psline[linecolor=black, linewidth=0.08](0.19711548,-9.782885)(4.9971156,-9.782885)(4.9971156,-14.582885)(0.19711548,-9.782885)(0.19711548,-14.582885)(4.9971156,-14.582885)(4.9971156,-14.582885)
\psline[linecolor=black, linewidth=0.08](4.9971156,-9.782885)(0.19711548,-14.582885)(0.19711548,-14.582885)
\psline[linecolor=black, linewidth=0.08](9.397116,-12.182884)(4.9971156,-9.782885)(4.9971156,-9.782885)
\psline[linecolor=black, linewidth=0.08](9.397116,-12.182884)(4.9971156,-14.582885)(4.9971156,-14.582885)
\psline[linecolor=black, linewidth=0.08](15.144653,-14.539504)(10.344803,-14.577384)(10.306923,-9.777533)(15.144653,-14.539504)(15.106773,-9.739654)(10.306923,-9.777533)(10.306923,-9.777533)
\psline[linecolor=black, linewidth=0.08](10.344803,-14.577384)(15.106773,-9.739654)(15.106773,-9.739654)
\psline[linecolor=black, linewidth=0.08](5.926,-12.212181)(10.344803,-14.577384)(10.344803,-14.577384)
\psline[linecolor=black, linewidth=0.08](5.926,-12.212181)(15.144653,-14.539504)(15.144653,-14.539504)
\psline[linecolor=black, linewidth=0.08](5.926,-12.212181)(10.306923,-9.777533)(10.306923,-9.777533)
\psline[linecolor=black, linewidth=0.08](5.926,-12.212181)(15.106773,-9.739654)(15.106773,-9.739654)
\psline[linecolor=black, linewidth=0.08](5.7971153,-12.182884)(9.397116,-12.182884)(9.397116,-12.182884)
\psline[linecolor=black, linewidth=0.08](0.19711548,-14.582885)(9.397116,-12.182884)(9.397116,-12.182884)
\psline[linecolor=black, linewidth=0.08](0.19711548,-9.782885)(9.397116,-12.182884)(9.397116,-12.182884)
\psdots[linecolor=black, dotsize=0.4](5.8771152,-12.202885)
\psdots[linecolor=black, dotsize=0.4](9.397116,-12.182884)
\psline[linecolor=black, linewidth=0.08](8.5971155,-12.182884)(9.397116,-12.182884)(9.397116,-12.182884)
\psdots[linecolor=black, dotsize=0.4](4.9971156,-9.782885)
\psdots[linecolor=black, dotsize=0.4](0.19711548,-9.782885)
\psdots[linecolor=black, dotsize=0.4](0.19711548,-14.582885)
\psdots[linecolor=black, dotsize=0.4](4.9971156,-14.582885)
\psdots[linecolor=black, dotsize=0.4](10.297115,-9.802884)
\psdots[linecolor=black, dotsize=0.4](15.0971155,-9.742885)
\psdots[linecolor=black, dotsize=0.4](15.117115,-14.522884)
\psdots[linecolor=black, dotsize=0.4](10.317116,-14.562884)
\end{pspicture}
}
		\end{center}
		\caption{Graph $P_2 \star K_4$} \label{P2starK4}
	\end{figure}

	\begin{remark}
The bounds in Theorem \ref{GstarH} are sharp.  For the upper  bound, it suffices to consider $G=P_2$ and $H=K_4$. Then 
$\big|O_{V(G)}\big|=0$, $\big|E_{V(G)}\big|=2$, $\big|O_{V(H)}\big|=0$, $\big|E_{V(H)}\big|=4$ and $\big|V(G)\big|=2$. Therefore we have $\gamma _{coe}(P_2 \star K_4)\leq2$. As we see in Figure \ref{P2starK4}, we have only two vertices with odd degree and those vertices are enough for our dominating set too. So 
$\gamma _{coe}(P_2 \star K_4)=2$ . For the lower bound, it suffices to consider $G=C_3$ and $H=K_4$. Then 
$\big|O_{V(G)}\big|=3$, $\big|E_{V(G)}\big|=0$, $\big|O_{V(H)}\big|=0$ and $\big|E_{V(H)}\big|=4$. One can easily check that 
$\gamma _{coe}(C_3 \star K_4)=12$.
	\end{remark}

\section{Co-even domination number of the  Haj\'{o}s sum of two graphs}

In this section we consider to Haj\'{o}s sum  of two graphs. 
Given graphs $G_1 = (V_1,E_1)$ and $G_2 = (V_2, E_2)$
with disjoint vertex sets, an edge $x_1y_1\in E_1$, and an edge $x_2y_2\in E_2$, the
Haj\'{o}s sum $G_3 = G_1(x_1y_1) +_H G_2(x_2y_2)$ is the graph obtained as follows:
begin with
$G_3' = (V_1 \cup V_2, E_1 \cup E_2)$; then in $G_3'$ delete edges $x_1y_1$ and
$x_2y_2$, identify vertices $x_1$ and $x_2$ as $v_H(x_1x_2)$, and add edge $y_1y_2$. Now define $G_3$ as
the current $G_3'$. Figure \ref{HaJ-K4C4} shows the Haj\'{o}s sum  of $K_4$ and $C_4$ with respect to $x_1y_1$ and
$x_2y_2$. First we state a sharp upper bound for Haj\'{o}s sum  of two graphs.

	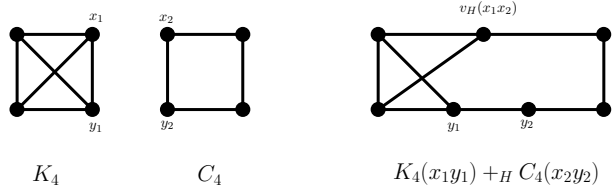
\begin{figure}
		\begin{center}
			\psscalebox{0.5 0.5}
{
\begin{pspicture}(0,-7.555)(15.994231,-2.665)
\rput[bl](2.1171153,-3.235){$x_1$}
\rput[bl](3.9771154,-3.255){$x_2$}
\rput[bl](2.1171153,-6.155){$y_1$}
\rput[bl](4.017115,-6.095){$y_2$}
\rput[bl](11.6371155,-6.135){$y_1$}
\rput[bl](13.577115,-6.055){$y_2$}
\rput[bl](11.957115,-3.055){$v_H(x_1x_2)$}
\psline[linecolor=black, linewidth=0.08](2.1971154,-3.555)(0.19711533,-3.555)(0.19711533,-5.555)(2.1971154,-5.555)(2.1971154,-3.555)(0.19711533,-5.555)(0.19711533,-5.555)
\psline[linecolor=black, linewidth=0.08](0.19711533,-3.555)(2.1971154,-5.555)(2.1971154,-5.555)
\psdots[linecolor=black, dotsize=0.4](0.19711533,-3.555)
\psdots[linecolor=black, dotsize=0.4](2.1971154,-3.555)
\psdots[linecolor=black, dotsize=0.4](0.19711533,-5.555)
\psdots[linecolor=black, dotsize=0.4](2.1971154,-5.555)
\psline[linecolor=black, linewidth=0.08](4.1971154,-5.555)(4.1971154,-3.555)(6.1971154,-3.555)(6.1971154,-5.555)(4.1971154,-5.555)(4.1971154,-5.555)
\psdots[linecolor=black, dotsize=0.4](4.1971154,-3.555)
\psdots[linecolor=black, dotsize=0.4](6.1971154,-3.555)
\psdots[linecolor=black, dotsize=0.4](6.1971154,-5.555)
\psdots[linecolor=black, dotsize=0.4](4.1971154,-5.555)
\psline[linecolor=black, linewidth=0.08](9.797115,-3.555)(11.797115,-5.555)(11.797115,-5.555)
\psdots[linecolor=black, dotsize=0.4](9.797115,-3.555)
\psdots[linecolor=black, dotsize=0.4](12.5971155,-3.555)
\psdots[linecolor=black, dotsize=0.4](9.797115,-5.555)
\psdots[linecolor=black, dotsize=0.4](11.797115,-5.555)
\psdots[linecolor=black, dotsize=0.4](15.797115,-3.555)
\psdots[linecolor=black, dotsize=0.4](15.797115,-5.555)
\psdots[linecolor=black, dotsize=0.4](13.797115,-5.555)
\psline[linecolor=black, linewidth=0.08](11.797115,-5.555)(13.797115,-5.555)(15.797115,-5.555)(15.797115,-3.555)(15.797115,-3.555)
\psline[linecolor=black, linewidth=0.08](15.797115,-3.555)(9.797115,-3.555)(9.797115,-5.555)(11.797115,-5.555)(11.797115,-5.555)
\psline[linecolor=black, linewidth=0.08](9.797115,-5.555)(12.5971155,-3.555)(12.5971155,-3.555)
\rput[bl](0.59711534,-7.555){\LARGE{$K_4$}}
\rput[bl](4.997115,-7.555){\LARGE{$C_4$}}
\rput[bl](10.197115,-7.555){\LARGE{$K_4(x_1y_1)+_H C_4(x_2y_2)$}}
\end{pspicture}
}
		\end{center}
		\caption{ Haj\'{o}s construction of $K_4$ and $C_4$ } \label{HaJ-K4C4}
	\end{figure}

	\begin{theorem}\label{Hajos}
Let $G_1=(V_1,E_1)$ and $G_2=(V_2,E_2)$ be two connected graphs with disjoint vertex sets, $x_1y_1\in E_1$ and $x_2y_2\in E_2$. Then for Haj\'{o}s sum 
$$G_3=G_1(x_1y_1)+_H G_2(x_2y_2),$$
we have:
$$\gamma _{coe}(G_3) \leq  \gamma _{coe} (G_1) + \gamma _{coe} (G_2) +1.  $$
	\end{theorem}

	\begin{proof}
Suppose that $D_{coe}(G_1)$ and $D_{coe}(G_2)$ be the co-even dominating sets of $G_1$ and $G_2$, respectively. We have 16 different cases for $x_1$, $y_1$, $x_2$ and $y_2$ regarding whether they are in $D_{coe}(G_1)$ and $D_{coe}(G_2)$ or not. Note that the degree of vertices $y_1$ and $y_2$ will remain the same in $G_3$ as their degree in $G_1$ and $G_2$, and the degree of $v_H(x_1x_2)$ will be odd if only one of $x_1$ or $x_2$ be odd. We only consider some cases as follows and the rest are similar:
\begin{itemize}
\item[(i)] 
$x_1,y_1\notin D_{coe}(G_1)$ and $x_2,y_2\notin D_{coe}(G_2)$. Since the degree of vertices $y_1$ and $y_2$ is even (otherwise they should be in the co-even dominating set), and they will remain even degree vertices in $G_3$ by the definition of Haj\'{o}s sum, and also the vertex $v_H(x_1x_2)$ in $G_3$ will remain even degree vertex in $G_3$ too, then 
 $$D_{coe}(G_1) \cup D_{coe}(G_2)$$ 
 is a co-even domination set for $G_3$, and 
$$\gamma _{coe}(G_3) \leq  \gamma _{coe} (G_1) + \gamma _{coe} (G_2).$$
\item[(ii)]
$x_1,y_1\in D_{coe}(G_1)$ and $x_2,y_2\in D_{coe}(G_2)$. Then 
$$\big( D_{coe}(G_1)-\{x_1\}\big) \cup \big( D_{coe}(G_2)-\{x_2\} \big) \cup \{v_H(x_1x_2)\}$$
is a co-even dominating set for  $G_3$, and 
$$\gamma _{coe}(G_3) \leq  \gamma _{coe} (G_1) + \gamma _{coe} (G_2)-1.$$
\item[(iii)] 
$x_1\in D_{coe}(G_1)$, $y_1\notin D_{coe}(G_1)$  and $x_2,y_2\notin D_{coe}(G_2)$. Now we consider to $D_{coe}(G_1) \cup D_{coe}(G_2)$ and show that it might not be a co-even dominating set for $G_3$. By the definition of Haj\'{o}s sum, $y_1$ is adjacent to $y_2$ and not adjacent to $v_H(x_1x_2)$. If the only vertex in $D_{coe}(G_1)$ be $x_1$ and is adjacent to $y_1$ in $G_1$, then there are no vertices in $D_{coe}(G_1) \cup D_{coe}(G_2)$ in $G_3$ and adjacent to $y_1$. Therefore we need at least one more vertices which is adjacent to $y_1$ in our dominating set or $y_1$ in our dominating set. One can easily check that
 $$D_{coe}(G_1) \cup D_{coe}(G_2) \cup \{y_2\} $$
 is a co-even dominating set for  $G_3$, and 
$$\gamma _{coe}(G_3) \leq  \gamma _{coe} (G_1) + \gamma _{coe} (G_2)+1.$$
\item[(iv)] 
$x_1\notin D_{coe}(G_1)$, $y_1\in D_{coe}(G_1)$  and $x_2,y_2\notin D_{coe}(G_2)$. Now
$$D_{coe}(G_1) \cup D_{coe}(G_2)$$
is a co-even dominating set for  $G_3$, because $y_2\notin D_{coe}(G_2)$ and whatever was adjacent to $x_2$ in $D_{coe}(G_2)$, is now adjacent to $v_H(x_1x_2)$, and hence
$$\gamma _{coe}(G_3) \leq  \gamma _{coe} (G_1) + \gamma _{coe} (G_2).$$
\end{itemize}
As mentioned before, the other cases are similar and therefore we have the result.
	\qed
	\end{proof}

	\begin{figure}
		\begin{center}
			\psscalebox{0.5 0.5}
{
\begin{pspicture}(0,-7.299306)(19.202778,-0.29791656)
\psline[linecolor=black, linewidth=0.08](3.401389,-1.6993054)(3.401389,-3.6993055)(1.8013889,-4.8993053)(0.2013889,-3.6993055)(0.2013889,-1.6993054)(1.8013889,-0.49930543)(1.8013889,-0.49930543)
\psline[linecolor=black, linewidth=0.08](5.001389,-1.6993054)(6.601389,-1.6993054)(8.201389,-1.6993054)(8.201389,-1.6993054)
\psline[linecolor=black, linewidth=0.08](5.001389,-1.6993054)(5.001389,-3.6993055)(5.001389,-5.2993054)(5.001389,-5.2993054)
\psdots[linecolor=black, dotsize=0.4](5.001389,-5.2993054)
\psdots[linecolor=black, dotstyle=o, dotsize=0.4, fillcolor=white](5.001389,-1.6993054)
\psdots[linecolor=black, dotstyle=o, dotsize=0.4, fillcolor=white](5.001389,-3.6993055)
\psdots[linecolor=black, dotsize=0.4](6.601389,-1.6993054)
\psline[linecolor=black, linewidth=0.08](8.201389,-1.6993054)(8.201389,-5.2993054)(8.201389,-5.2993054)
\psdots[linecolor=black, dotstyle=o, dotsize=0.4, fillcolor=white](8.201389,-1.6993054)
\psdots[linecolor=black, dotstyle=o, dotsize=0.4, fillcolor=white](8.201389,-3.6993055)
\psdots[linecolor=black, dotsize=0.4](8.201389,-5.2993054)
\psdots[linecolor=black, dotsize=0.4](16.20139,-5.2993054)
\psline[linecolor=black, linewidth=0.08](19.001389,-1.6993054)(19.001389,-5.2993054)(19.001389,-5.2993054)
\psdots[linecolor=black, dotstyle=o, dotsize=0.4, fillcolor=white](19.001389,-3.6993055)
\psdots[linecolor=black, dotsize=0.4](19.001389,-5.2993054)
\psline[linecolor=black, linewidth=0.08](13.001389,-0.49930543)(11.401389,-1.6993054)(11.401389,-3.6993055)(13.001389,-4.8993053)(14.601389,-3.6993055)(14.601389,-3.6993055)
\psline[linecolor=black, linewidth=0.08](19.001389,-1.6993054)(17.401388,-1.6993054)(17.401388,-1.6993054)
\psdots[linecolor=black, dotsize=0.4](3.401389,-1.6993054)
\psdots[linecolor=black, dotsize=0.4](3.401389,-0.49930543)
\psdots[linecolor=black, dotsize=0.4](0.2013889,-3.6993055)
\psdots[linecolor=black, dotstyle=o, dotsize=0.4, fillcolor=white](0.2013889,-1.6993054)
\psdots[linecolor=black, dotstyle=o, dotsize=0.4, fillcolor=white](3.401389,-3.6993055)
\psdots[linecolor=black, dotstyle=o, dotsize=0.4, fillcolor=white](1.8013889,-4.8993053)
\psline[linecolor=black, linewidth=0.08](16.20139,-5.2993054)(16.20139,-3.6993055)(16.20139,-4.0993056)
\psdots[linecolor=black, dotsize=0.4](15.401389,-1.6993054)
\psdots[linecolor=black, dotsize=0.4](15.401389,-0.49930543)
\psdots[linecolor=black, dotstyle=o, dotsize=0.4, fillcolor=white](19.001389,-1.6993054)
\psdots[linecolor=black, dotstyle=o, dotsize=0.4, fillcolor=white](11.401389,-1.6993054)
\psdots[linecolor=black, dotstyle=o, dotsize=0.4, fillcolor=white](13.001389,-4.8993053)
\psline[linecolor=black, linewidth=0.08](1.8013889,-0.49930543)(3.401389,-1.6993054)(3.401389,-0.49930543)(3.401389,-0.49930543)
\psdots[linecolor=black, dotstyle=o, dotsize=0.4, fillcolor=white](1.8013889,-0.49930543)
\psline[linecolor=black, linewidth=0.08](14.601389,-3.6993055)(16.20139,-3.6993055)(16.20139,-3.6993055)
\psdots[linecolor=black, dotstyle=o, dotsize=0.4, fillcolor=white](16.20139,-3.6993055)
\psdots[linecolor=black, dotstyle=o, dotsize=0.4, fillcolor=white](14.601389,-3.6993055)
\psline[linecolor=black, linewidth=0.08](13.001389,-0.49930543)(15.401389,-1.6993054)(17.401388,-1.6993054)(17.401388,-1.6993054)
\psline[linecolor=black, linewidth=0.08](15.401389,-1.6993054)(15.401389,-0.49930543)(15.401389,-0.49930543)
\psdots[linecolor=black, dotstyle=o, dotsize=0.4, fillcolor=white](17.401388,-1.6993054)
\psdots[linecolor=black, dotstyle=o, dotsize=0.4, fillcolor=white](13.001389,-0.49930543)
\psdots[linecolor=black, dotstyle=o, dotsize=0.4, fillcolor=white](11.401389,-3.6993055)
\rput[bl](2.7613888,-2.0393054){$x_1$}
\rput[bl](5.201389,-2.0793054){$x_2$}
\rput[bl](2.801389,-3.6793053){$y_1$}
\rput[bl](5.2813888,-3.8193054){$y_2$}
\rput[bl](1.4013889,-7.2993054){\LARGE{$G_1$}}
\rput[bl](6.201389,-7.2993054){\LARGE{$G_2$}}
\rput[bl](11.801389,-7.2993054){\LARGE{$G_3=G_1(x_1y_1)+_H G_2(x_2y_2)$}}
\rput[bl](14.421389,-4.2793055){$y_1$}
\rput[bl](15.701389,-4.2393055){$y_2$}
\rput[bl](14.761389,-2.3593054){$v_H(x_1x_2)$}
\end{pspicture}
}
		\end{center}
		\caption{Graphs $G_1$, $G_2$ and their  Haj\'{o}s sum, respectively. } \label{HaJ-fig}
	\end{figure}
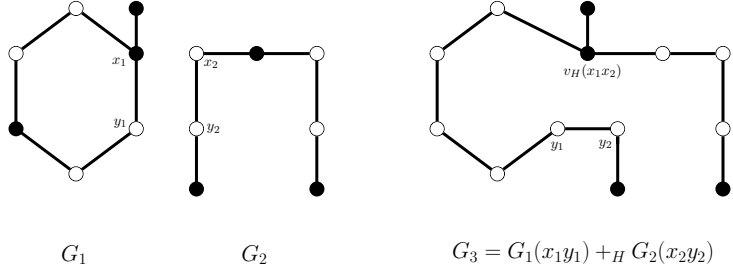

	\begin{remark}
The upper bound in Theorem \ref{Hajos} is sharp. Consider to the graphs $G_1$ and $G_2$ in Figure \ref{HaJ-fig}. The set of black vertices in $G_1$ and $G_2$ are co-even dominating sets for these graphs. Now for $G_3=G_1(x_1y_1)+_H G_2(x_2y_2)$, the set of black vertices is a subset of co-even dominating set, because they have odd degrees. Now among white vertices of $G_3$, it is easy to see  that we need at least three vertices to have a co-even dominating set. Therefore, $\gamma _{coe}(G_3) =  \gamma _{coe} (G_1) + \gamma _{coe} (G_2) +1$.
	\end{remark}

In the last theorem, we have found an upper bound for Haj\'{o}s sum of two graphs. We believe that the following is a lower bound for that:

	\begin{conjecture}\label{Conj}
Let $G_1=(V_1,E_1)$ and $G_2=(V_2,E_2)$ be two connected graphs with disjoint vertex sets, $x_1y_1\in E_1$ and $x_2y_2\in E_2$. Then for Haj\'{o}s sum 
$$G_3=G_1(x_1y_1)+_H G_2(x_2y_2),$$
we have:
$$\gamma _{coe}(G_3) \geq \gamma _{coe} (G_1) + \gamma _{coe} (G_2)-2.  $$
	\end{conjecture}

	\begin{remark}
If the Conjecture \ref{Conj} be true, the lower bound for Haj\'{o}s sum of two graphs is sharp. It suffices to consider $G_1$ and $G_2$ as complete graph $K_4$. Then for every $x_1y_1, x_2y_2\in E_{K_4}$, we have 
$\gamma _{coe}\big( K_4(x_1y_1)+_H K_4(x_2y_2) \big)=6$.
	\end{remark}

\section{Conclusions}

In this paper, we obtained the co-even domination number of join and corona of two graphs and presents some sharp lower and upper bounds for neighbourhood corona and  Haj\'{o}s sum of two graphs
Future topics of interest for future research include the following suggestions:

\begin{itemize}
\item[(i)]
Proving Conjecture \ref{Conj} or finding the correct lower bound for $\gamma _{coe}(G_1(x_1y_1)+_H G_2(x_2y_2))$.
\item[(ii)]
Finding co-even domination number of other binary operations of graphs such as lexicographic product, strong product, tensor product, etc.
\item[(iii)]
Finding co-even domination number of unary operations on graphs. 
\end{itemize}

	\section{Acknowledgements} 
	
	The  author would like to thank the Research Council of Norway and Department of Informatics, University of
	Bergen for their support.	Also he is thankful to Michael Fellows and
Saeid Alikhani for conversations and sharing their pearls of wisdom with him during the course
of this research.

\end{document}